\documentclass[a4paper,12pt]{amsart}

\usepackage{latexsym,amsmath,amsfonts,amscd,amssymb,amsthm,mathrsfs}
\usepackage{tikz}
\usepackage{caption}
\usepackage{enumerate}
\usepackage[all]{xy}
\usepackage{color}
\usepackage{soul}
\usepackage{pdflscape}
\usepackage{booktabs}
\usepackage{longtable}
\usepackage{footmisc}
\usepackage{multirow}
\usepackage{comment}
\usepackage{fullpage}
\usepackage{float}

\usepackage{ cleveref}

\def\z{\mathfrak{z}}
\def\u{\mathfrak{u}}

\def\g{\mathfrak{g}}
\def\h{\mathfrak{h}}
\def\n{\mathfrak{n}}

\def\R{\mathbb{R}}

\def\Z{\mathbb{Z}}
\def\N{\mathbb{N}}

\def\ad{\operatorname{ad}}
\def\tr{\operatorname{tr}}
\def\I{\operatorname{Id}}
\def\alt{\raise1pt\hbox{$\bigwedge$}}
\def\pint{\langle \cdotp,\cdotp \rangle }

\theoremstyle{plain}
\newtheorem{thm}{\bf Theorem}[section]
\newtheorem*{thm*}{\bf Theorem}
\newtheorem{cor}[thm]{\bf Corollary}
\newtheorem{prop}[thm]{\bf Proposition}

\theoremstyle{definition}

\newtheorem{ejemplo}[thm]{\bf Example}

\theoremstyle{remark}
\newtheorem{obs}[thm]{\bf Remark}

\title{Locally conformal symplectic structures on Lie algebras of type I and their solvmanifolds}

\author{M. Origlia}
\email{origlia@famaf.unc.edu.ar}

\date{}
\address{KU Leuven Kulak, E. Sabbelaan 53, BE-8500 Kortrijk, Belgium and \\
FaMAF-UNC, CIEM-CONICET, Ciudad Universitaria, 5000 C\'{o}rdoba, Argentina}
\thanks{}


\begin{document}

\begin{abstract}
We study Lie algebras of type I, that is, a Lie algebra $\g$ where all the eigenvalues of the operator $\ad_X$ are imaginary for all $X\in\g$. We prove that the Morse-Novikov cohomology of a Lie algebra of type I is trivial for any closed  $1$-form. We focus on  locally conformal symplectic structures (LCS) on Lie algebras of type I. In particular we show that for a Lie algebra of type I any LCS structure is of the first kind. We also exhibit lattices for some $6$-dimensional Lie groups of type I admitting left invariant LCS structures in order to produce compact solvmanifolds equipped with an invariant LCS structure.
\end{abstract}
\maketitle

\section{Introduction}
The most important class of Hermitian manifolds is definitely the class of K\"ahler manifolds. Given a smooth manifold, there are topological obstructions to the existence of a K\"ahler metric and it is known that many important manifolds cannot admit K\"ahler metrics. 
A bigger class of Hermitian manifolds is given by the locally conformal K\"ahler (LCK) manifolds, which have shown to be of great importance lately.
These are Hermitian manifolds such that each point has a neighborhood where the metric is conformal to a K\"ahler metric. Equivalently, a Hermitian manifold $(M,J,g)$ is LCK if and only if there exists a closed $1$-form $\theta$ such that $d\omega=\theta\wedge\omega$
where $\omega$ is the fundamental 2-form. In this case, the $1$-form $\theta$ is called the Lee form. LCK manifolds were introduced by I. Vaisman in \cite{V} and deeply studied by many others authors since then.

Now, if we only consider the equation $d\omega=\theta\wedge\omega$ for a non degenerate $2$-form $\omega$ and a closed $1$-form $\theta$ and do not assume the existence of a Hermitian structure, we arrive at the notion of a locally conformal symplectic (LCS) manifold.  These manifolds were considered by Lee in \cite{L} and they have been firstly studied by Vaisman in \cite{V}. Some recent results can be found in \cite{Ba,Ha,HR,HR2,LV}, among others. 
%LCS manifolds play an important role in mathematics and in physics as well, for example in Hamiltonian mechanics, generalizing the usual description of the phase space in terms of symplectic geometry. 
This topic is very active, for instance, in \cite{AD} the authors proved that certain compact complex surfaces admit a LCS structure. This is an important step towards showing that every complex surface (except Belgun's counterexamples of some Inoue surfaces) carries a LCK metric. 
LCS manifolds are also very important in theoretical physics, in particular locally conformal symplectic structures play an important role in Hamiltonian mechanics, generalizing the usual description of the phase space in terms of symplectic geometry. 
Indeed, the phase space of a Hamiltonian system is the cotangent bundle $T^*M$ of a manifold $M$ which parametrizes the positions of the physical system. The cotangent bundle $T^*M$ is in a natural way a symplectic manifold, where the symplectic $2$-form $\omega$ is given by the differential of the tautological $1$-form on $T^*M$. Since $\omega$ is non degenerate, then the Hamiltonian vector field is uniquely determined by $\omega$. The fact that non-degeneracy is a local condition implies that the definition of the Hamiltonian vector field is local and therefore locally conformal symplectic manifolds provide an adequate and more general context for Hamiltonian mechanics. It can be seen that the cotangent bundle $T^*M$ admits a canonical exact locally conformal symplectic structure (see for instance \cite{Baz1}, \cite{HR2}).

There is a way to distinguish LCS structures on a manifold, since they can be of the first or of the second kind in the sense of Vaisman (see \cite{V}). To do this distinction one considers infinitesimal automorphisms of $(\omega,\theta)$, i.e. $X\in\mathfrak{X}(M)$ such that $\textrm L_X\omega=0$ where $\textrm L$ denotes the Lie derivative. This implies $\textrm L_X\theta=0$ as well and therefore $\theta(X)$
is a constant function on $M$. If there exists a infinitesimal automorphism $X$ with $\theta(X)\neq0$, the LCS structure is said to be of the first kind, and it is of the second kind otherwise.

In this work we focus on left invariant LCS structures on Lie groups or equivalently LCS structures on their Lie algebras. We also study the existence of lattices in these Lie groups in order to obtain compact examples of manifolds admitting LCS structures.  

In \cite{BM} it was shown that any LCS structure on a nilpotent Lie algebra is of the first kind and they proved that certain Lie algebras with a LCS structure of the first kind are a double extension of a symplectic Lie algebra. 

In this work we consider a larger class of Lie algebras which includes the nilpotent ones, namely the Lie algebras of type I. Recall that a Lie algebra $\g$ is said to be of type I if all the eigenvalues of the operator $\ad_X$ are imaginary for all  $X\in\g$ (see \cite{OnVi}).
Note also that a Lie algebra $\g$ of type I is in particular unimodular, which is a necessary condition for the associated simply connected Lie group to admit lattices according to a result of \cite{Mil}. Recall that a Lie algebra is unimodular if $\tr(\ad_X)=0$ for any $X\in\g$.
Our aim in this work is to study Lie algebras of type I equipped with LCS structures and the existence of lattices in the associated simply connected Lie groups.

%In this direction, we prove that any LCS structure on a type I Lie algebra is of the first kind (see Corollary \ref{ImgPuras-1tipo}). But we also obtain a more general result, we prove that any LCS structure on a larger class of unimodular Lie algebras is of the first kind (see Theorem \ref{teo1}).

\medskip
 
The outline of this article is as follows. 
In Section $2$ we recall some definitions and known results about LCS structures on manifolds and on Lie algebras. In particular we recall some results of \cite{BM} about LCS structures of the first kind on Lie algebras. 
In Section $3$ we study Lie algebras of type I and we prove our first result about the Morse-Novikov cohomology for a Lie algebra of type I (see Corollary \ref{tipoI-coho-trivial}). Then we focus on LCS structures on these Lie algebras and we prove that any Lie algebra of type I admits only LCS structures of the first kind (see Corollary \ref{ImgPuras-1tipo}).
We determine all the $4$-dimensional Lie algebras of type I and we also determine all the $5$-dimensional Lie algebras of type I admitting a contact structure. Then we use this classification to show examples of $6$-dimensional Lie algebras of type I admiting a LCS structure. 
Finally in Section $4$ we exhibit lattices in the simply connected Lie groups associated to some of these Lie algebras in order to produce compact solvmanifolds equipped with invariant LCS structure.

\

\noindent \textbf{Acknowledgements.} This work was partially supported by CONICET, SECyT-UNC (Argentina) and the Research Foundation Flanders (Project G.0F93.17N). I would like to thank A.~Andrada for his interesting comments and useful suggestions on the first versions of this paper.

\

\section{Preliminaries}

\medskip

A \textit{locally conformal symplectic} 
structure (LCS for short) on the manifold $M$ is a non degenerate $2$-form $\omega$ such that 
there exists an open cover $\{U_i\}$ and smooth functions $f_i$ on $U_i$ such that 
\[\omega_i=\exp(-f_i)\omega\] 
is a symplectic form on $U_i$, i.e. $d\omega_i=0$. This condition is equivalent to requiring that
\begin{equation}\label{lcs}
d\omega=\theta\wedge\omega
\end{equation} 
for some closed $1$-form $\theta$, called the Lee form. 
Moreover, $M$ is called globally conformal symplectic (GCS) if there exist  a $C^{\infty}$ function, $f:M\to\R$, such that $\exp(-f)g$ is a symplectic form. Equivalently, $M$ is a GCS 
manifold if there exists a exact $1$-form $\theta$ globally defined on $M$ such that 
$d\omega=\theta\wedge\omega$.   
%We note that the Lee form is also uniquely determined by the non degenerate $2$-form $\omega$, but there is not an explicit formula for $\theta$. 
The pair $(\omega, \theta)$ will be called a LCS structure on $M$.

It is well known that
\begin{itemize}
	\item If $(\omega,\theta)$ is a LCS structure on $M$, then $\omega$ is symplectic if and only if 
	$\theta=0$. Indeed, $\theta\wedge\omega=0$ and $\omega$ non degenerate imply $\theta=0$. 
	%Accordingly, a LCK structure $(J,g)$ is K\"ahler if and only if $\theta=0$.
	\item The same argument shows that $\theta$ is uniquely determined by equation \eqref{lcs}, but there is not an explicit formula for the Lee form.
	\item If $\omega$ is a non degenerate $2$-form on $M$, with $\dim M\ge 6$, such that \eqref{lcs}
	holds for some $1$-form $\theta$ then $\theta$ is automatically closed and therefore $M$ is LCS.
\end{itemize}

\

If we add a Hermitian structure $(J,g)$ on $M$ compatible with the LCS $2$-form $\omega$, i.e. $\omega(\cdot,\cdot)=g(J\cdot,\cdot)$, we arrive to the notion of {\em locally conformal K\"ahler} (LCK
for short) structure on $M$, where $J$ is a complex structure and $g$ is a Hermitian metric. Equivalently, $(M,J,g)$ is a locally conformal K\"ahler manifold if there exists an open covering $\{ U_i\}_{i\in I}$ of $M$ and a family $\{ f_i\}_{i\in I}$ of $C^{\infty}$ functions, $f_i:U_i \to \R$, such that each local metric 
%\begin{equation}\label{gi} 
\[g_i=\exp(-f_i)\,g|_{U_i} \]
%\end{equation} 
is K\"ahler. 
%Also $(M,J,g)$ is {\em globally conformal K\"ahler} (GCK) if there exists a $C^{\infty}$ function, $f:M\to\R$, such that the metric $\exp(-f)g$ is K\"ahler.
%An equivalent characterization of a LCK manifold can be given in terms of the fundamental form $\omega$, which is defined by $\omega(X,Y)=g(JX,Y)$, for all $X,Y \in \mathfrak{X}(M)$. Indeed, a Hermitian manifold $(M,J,g)$ is LCK if and only if there exists a closed $1$-form $\theta$ globally defined on $M$ such that \[d\omega=\theta\wedge\omega.\] This closed $1$-form $\theta$ is called the \textit{Lee form} (see \cite{L}). 
The Lee form $\theta$ is completely determined by $\omega$, and in this case there is an explicit formula given by
%\begin{equation}\label{lee}
\[\theta=-\frac{1}{n-1}(\delta\omega)\circ J, \]
%\end{equation} 
where $\delta$ is the codifferential and $2n$ is the dimension of $M$.
There is a distinguished class of LCK manifolds satisfying that the Lee form is parallel with respect to the Levi Civita connection of the Hermitian metric $g$, namely Vaisman manifolds which have been much studied by Vaisman (see \cite{V2,V3}) and recently by many authors, see for instance \cite{Ov, GMO}.

%An interesting problem is try to find examples of manifolds admitting LCS structures which not carry on LCK structures (see \cite{BK, BM}).  

\

Returning to the class of LCS manifolds, we recall next a definition due to Vaisman (see \cite{V}) about two different types of LCS structures. If $(\omega, \theta)$ is a LCS 
structure on $M$, a vector field $X$ is called an infinitesimal automorphism of 
$(\omega,\theta)$ if $\textrm{L}_X\omega=0$, where $\textrm{L}$ denotes the Lie derivative. This 
implies $\textrm{L}_X\theta=0$ as well and, as a consequence, $\theta(X)$ is a constant function on $M$. We consider  $\chi_\omega(M)=\{X\in\mathfrak{X}(M): \textrm{L}_X\omega=0\}$ which is a Lie subalgebra of $\mathfrak{X}(M)$, then the map $\theta|_{\chi_\omega(M)} : \chi_\omega(M) \to \R$ is a well defined Lie algebra morphism called the Lee morphism. 
If there exists an infinitesimal automorphism $X$ such that $\theta(X)\neq 0$, the LCS structure $(\omega,\theta)$ is said to be of {\em the first kind}, and it is of {\em the second kind} otherwise. This condition is equivalent to verifing whether the Lee morphism is surjective or identically zero.

There is more information about LCS structures of the first kind, for example, in \cite{V} interesting relations with contact geometry are shown and it is proved that a manifold with a LCS structure of the first kind admits distinguished foliations.

\smallskip

There is another way to distinguish LCS structures in terms of a suitable cohomology. In order to do this,
one can deform the de Rham differential $d$ to obtain the adapted 
differential operator 
\[d_\theta \alpha= d\alpha -\theta\wedge\alpha,\]
for any differential form $\alpha \in \Omega^*(M)$.
Since $\theta$ is $d$-closed, this operator satisfies $d_\theta^2=0$, thus it defines the adapted cohomology 
$H_\theta^*(M)$ of $M$ relative to the closed $1$-form $\theta$, known as the {\em Morse-Novikov cohomology}. If $\theta$ is exact, it is easy to see that $H_\theta^*(M)\cong H_{dR}^*(M)$. It is known that if $M$ is a compact oriented $n$-dimensional manifold, then 
$H_\theta^0(M)= H_\theta^n(M)=0$ for any non exact closed $1$-form $\theta$ (see for instance 
\cite{GL,Ha}). It is also known that for any compact Vaisman manifold the Morse-Novikov cohomology is trivial (see \cite{LLMP}). For any LCS structure $(\omega,\theta)$ on $M$, the $2$-form $\omega$ 
defines a cohomology class $[\omega]_\theta\in H_\theta^2(M)$, since 
$d_\theta\omega=d\omega-\theta\wedge \omega=0$. Because of that, it is natural to study the Morse-Novikov cohomology for a LCS manifold with respect to its Lee form.
The LCS structure $(\omega,\theta)$ is said to be {\em exact} if $\omega$ is $d_\theta$-exact or $[\omega]_\theta =0$, i.e. $\omega= d\eta-\theta\wedge \eta$ for some $1$-form $\eta$, and it is {\em non-exact} if $[\omega]_\theta \neq0$. 
It was proved in \cite{V} that if the LCS structure $(\omega,\theta)$
is of the first kind on $M$ then $\omega$ is $d_\theta$-exact, i.e. $[\omega]_\theta=0$. The converse, however, need not be true.

\

In this work we focus on left invariant LCS structures on Lie groups. Recall that a LCS structure $(\omega,\theta)$ on a Lie group $G$ is called left invariant 
if $\omega$ is left invariant, which easily implies that $\theta$ is also left invariant using condition \eqref{lcs}. 
Accordingly, we say that a Lie algebra $\g$ admits a {\em locally conformal symplectic} (LCS) 
structure if there exist $\omega\in\alt^2\g^*$ and $\theta \in \g^*$, with $\omega$ non degenerate 
and $\theta$ closed, such that \eqref{lcs} is satisfied. Since $\omega$ is a non degenerate $2$-form, we have an isomorphism $\g\to\g^*$ given by $X\to i_X \omega$. Therefore there exists a distinguished element $V\in\g$ such that $i_V \omega=\theta$, this vector is called the \text{Lee vector} of the LCS structure $(\omega,\theta)$ following the notation of \cite{BM}.

As in the case of manifolds we have that a LCS structure $(\omega,\theta)$ on a Lie algebra $\g$ 
can be of the first kind or of the second kind. Indeed, let us denote by $\g_\omega$ the set of 
infinitesimal automorphisms of the LCS structure, that is, 
\begin{equation}\label{autom}
\g_\omega = \{x\in\g: \textrm{L}_x\omega=0\} = \{x\in\g: \omega([x,y],z)+\omega(y,[x,z])=0 \; \text{for all} \; y,z\in\g\}.
\end{equation}
Note that $\g_\omega \subset \g$ is a Lie subalgebra, thus the restriction of $\theta$ to 
$\g_\omega$ is a Lie algebra morphism called {\em Lee morphism.} The LCS structure $ (\omega,\theta)$ is said to be {\em of the first kind} if the Lee morphism is surjective, and {\em of the second kind} if it is identically zero (see \cite{BM}).

\smallskip

For a Lie algebra $\g$ and a closed $1$-form $\theta\in\g^*$ we also have the Morse-Novikov cohomology 
$H_\theta^*(\g)$ defined by the differential operator \[d_\theta \alpha= d\alpha 
-\theta\wedge\alpha,\] 
on $\alt^* \g^*$. According to \cite{Mil}, this Morse-Novikov cohomology coincides with the Lie algebra 
cohomology of $\g$ with coefficients in a $1$-dimensional $\g$-module $V_{\theta}$, where the 
action of $\g$ on $V_{\theta}$ is given by
\begin{equation}\label{act}
Xv=-\theta(X)v, \quad X\in\g,\, v\in V_{\theta}.
\end{equation}
The fact that $\theta$ is closed guarantees that this is a Lie algebra representation. 

As in manifolds, we have that a LCS structure $(\omega,\theta)$ on a Lie algebra is said to be exact if $[\omega]_\theta =0$ or non exact if $[\omega]_\theta\neq0$. 

%\footnote{Se puede agregar algo mas de las cohomologia adaptada si hace falta}

\smallskip

%In this work we study LCS structures on Lie algebras, or equivalently left invariant LCS structures on Lie groups. 
Returning to Lie groups we know that on a simply connected Lie group any left invariant LCS structure turns 
out to be globally conformal to a symplectic structure, which is equivalent to having a symplectic structure on the Lie group. Therefore we will study 
compact quotients of such a Lie group by discrete subgroups, which will be non 
simply connected and will inherit a 'strict' LCS structure. Recall that a discrete subgroup 
$\Gamma$ of a simply connected Lie group $G$ is called a \textit{lattice} if the quotient 
$\Gamma\backslash G$ is compact.  The quotient $\Gamma\backslash G$ is known as a solvmanifold if 
$G$ is solvable and as a nilmanifold if $G$ is nilpotent, and in these cases we have that 
$\pi_1(\Gamma\backslash G)\cong \Gamma$. 

We note that a LCS structure of the first kind on a Lie 
algebra $\g$ induces a LCS structure of the first kind on any solvmanifold $\Gamma\backslash G$ with $\text{Lie}(G)=\g$.
%compact quotient of the corresponding simply connected Lie group by a discrete subgroup. 
  
%Hattori proved in \cite{Hat} that if $V$ is a finite dimensional triangular\footnote{A $\g$-module $V$ is called triangular if the endomorphisms of $V$ defined by $v\mapsto Xv$ have only real 	eigenvalues for any $X\in\g$.} $\g$-module, then $\overline{V}:=C^\infty(\Gamma\backslash G)\otimes V$ is a $\mathfrak{X}(\Gamma\backslash G)$-module and there is an isomorphism \begin{equation}\label{kill} H^*(\g,V) \cong H^*(\mathfrak{X}(\Gamma\backslash G), \overline{V}).\end{equation} Therefore: \begin{itemize}
%	\item If $V=\R$ is the trivial $\g$-module, then the right-hand side in \eqref{kill} gives the usual de Rham cohomology of $\Gamma\backslash G$, so that \begin{equation}\label{deRham}	H^*(\g) \cong H^*_{dR}(\Gamma\backslash G).	\end{equation}
%\item If $V=V_\theta$ with the action given by \eqref{act}, then we can identify $\overline{V}$ with $C^\infty(\Gamma\backslash G)$ and the action of $\mathfrak{X}(\Gamma\backslash G)$ on $C^\infty(\Gamma\backslash G)$ is given by \[ X\cdot f= Xf-\theta(X)f, \qquad X\in\g, \, f\in C^\infty(\Gamma\backslash G).\]	Here we are using that there is a natural inclusion $\g \hookrightarrow 	\mathfrak{X}(\Gamma\backslash G)$ and a bijection $C^\infty(\Gamma\backslash G)\otimes \g \to 	\mathfrak{X}(\Gamma\backslash G)$ given by $f\otimes X \mapsto fX$. As a consequence, in this case 	\eqref{kill} becomes (cf. \cite[Corollary 4.1]{Mil})	\begin{equation}\label{kill_adapted}	H^*_\theta(\g)\cong H^*_\theta(\Gamma\backslash G).	\end{equation}\end{itemize}

\

\subsection{Lie algebras with LCS structures of the first kind}

Let $\g$ a Lie algebra with a LCS structure $(\omega, \theta)$ of the first kind. Then, according to \cite{BM}, $\omega$ is $d_\theta$-exact, that is, there exists a $1$ form $\eta$ such that $d_\theta\eta=d\eta-\theta\wedge\eta=\omega$. There is a distinguished vector $U\in\g$ such that $\eta=-i_U\omega$, this vector is called the anti-Lee vector.

Lie algebras equipped with LCS structures of the first kind were studied recently in \cite{BM}. We briefly recall some results. Firstly, it is proved that if $\g$ is a unimodular Lie algebra with an exact LCS structure $(\omega,\theta)$ then $(\omega,\theta)$ is of the first kind. As we mentioned above the converse is true, due to Vaisman. Therefore we have 

\begin{prop}[\cite{BM}]\label{1º_exacta}
	If $\g$ is a unimodular Lie algebra with a LCS structure $(\omega,\theta)$, then $(\omega,\theta)$ is of the first kind if and only if $(\omega,\theta)$ is exact.
\end{prop}
In \cite{BM} it is also established a relation between LCS structures of the first kind and both contact and symplectic structures. Recall that a $(2n+1)$-dimensional Lie algebra $\h$ is a contact Lie algebra if it admits a contact structure, that is, a $1$-form $\eta$ such that $\eta\wedge (d\eta)^n\neq0$. This form $\eta$ is called the contact form and the unique vector $V\in\h$ satisfying $\eta(V)=1$ and $i_Vd\eta=0$ is called the Reeb vector of the contact structure.

\begin{prop}[\cite{BM}]\label{lcs_contact}
	There is a one to one correspondence between contact Lie algebras $(\h,\eta)$ of dimension $(2n+1)$ endowed with a derivation $D$ such that $\eta\circ D=0$ and $(2n+2)$-dimensional LCS Lie algebras  of the first kind $(\g,\omega,\theta)$. The relation is given by $\h=\ker\theta$, $\omega=d_\theta \eta$ and $D=\ad_U$ where $U$ is the anti-Lee vector.
\end{prop}

It is well known that if $(\h,\eta)$ is a contact Lie algebra then the center $\z(\h)$ of $\h$ has dimension at most $1$, furthermore if $\z(\h)$ is not trivial then it is generated by the Reeb vector, that is, $\z(\h)=\text{span}\{V\}$.
For a Lie algebra $\g$ with a LCS structure of the first kind where the Lee vector $V$ is central the following result is known 

\begin{prop}[\cite{BM}]\label{lcs_symplectic}
	There is a one to one correspondence between Lie algebras of dimension $2n +2$ admitting a LCS structure of the first kind with central Lee vector and symplectic Lie algebras $(\mathfrak{s}, \beta)$ of dimension $2n$ endowed with a derivation $E$ such that $\beta(EX,Y)+\beta(X,EY)=0$ for all $X,Y\in\mathfrak s$.
\end{prop}

We explain briefly the correspondence in Proposition \ref{lcs_symplectic}. Let $\g$ be a $(2n+2)$-dimensional Lie algebra with a LCS structure of the first kind $(\omega, \theta)$ with central Lee vector. It follows from Proposition \ref{lcs_contact} that $(\ker\theta,\eta)$ is a contact Lie algebra. It can be seen that $\mathfrak s=\ker\theta$ is a Lie algebra with Lie bracket $[\cdot,\cdot]_{\mathfrak s}$ given by $$[X,Y]=d\eta(X,Y)V + [X,Y]_{\mathfrak s}$$
for all $X,Y\in\mathfrak s$, where $V$ is the Lee vector and $[X,Y]_{\mathfrak s}\in\mathfrak s$ denotes the component of $[X,Y]$ in $\mathfrak s$.
It is clear that $(\ker\eta, [\cdot,\cdot]_{\mathfrak s}, d\eta)$ is a symplectic Lie algebra and if $U$ is the anti-Lee vector, then $E=\ad_U|_{\ker\eta}$ is a derivation of $\ker\eta$ satisfying $\beta(EX,Y)+\beta(X,EY)=0$ for all $X,Y\in\mathfrak s$. Any derivation satisfying this condition is called a \textit{symplectic derivation}.

Conversely, let $(\mathfrak s, [\cdot,\cdot]_{\mathfrak s}, \beta)$ be a symplectic Lie algebra endowed with a derivation $E$ such that $\beta(EX,Y)+\beta(X,EY)=0$ for all $X,Y\in\mathfrak s$. 
%Let $\h$ be the central extension of $\mathfrak s$ given by the $2$-cocycle $\beta$, that is, $\h=\R V\oplus\mathfrak s$ with Lie bracket given by $$[X,Y]=\beta(X,Y)V + [X,Y]_{\mathfrak s}$$. 
%Let us consider $\g=\R U\ltimes_D \h$ where $D|_\mathfrak s=E$ and $D(V)=0$. It is easy to check that $D$ is a derivation of $\h$. 
Consider now the Lie algebra $\g$ as the double extension of $(\mathfrak s, [\cdot,\cdot]_{\mathfrak s}, \beta, E)$, that is, $\g$ is given by $\g=\R U\oplus\R V\oplus\mathfrak s$ with Lie bracket defined by
$[X,Y]=\beta(X,Y)V + [X,Y]_{\mathfrak s}$ and $[U,X]=EX$, for all $X,Y\in\mathfrak s$. We define the $1$-forms $\theta,\eta\in\g^*$ by $\theta(U)=1, \theta(V)=0, \eta(V)=1, \eta(U)=0$ and $\theta(X)=\eta(X)=0$ for all $X\in\mathfrak s$. Then $\omega=d\eta-\theta\wedge\eta$ defines a LCS structure of the first kind on $\g$ and the Lee vector $V$ is central. In this case we will say that $\g$ is the double extension of $\mathfrak s$ by the pair $(E,\beta)$ (see \cite{ARS} for more details).
 
\

\subsection{Lie algebras of type I}
Recall that a Lie algebra $\g$ is said to be of \textit{type I} if all the eigenvalues of the operator $\ad_X$ are imaginary for all $X\in\g$ (some of them may be equal to zero). In \cite{AHG} these Lie algebras are called algebras of type R, an abbreviation of rigid. A Lie group is called a group of type I if its Lie algebra is of type I. A connected Lie group $G$ is characterized by the fact that all eigenvalues of the operator $\operatorname{Ad}_g$ have absolute value equal to $1$ for all $g\in G$.

This class of Lie algebras is in some sense opposite to the class of completely solvable Lie algebras. Recall that Lie algebra $\g$ is said to be of completely solvable if all the eigenvalues of the operator $\ad_X$ are real for all $X\in\g$.  Note that any nilpotent Lie algebra is of type I. 
Moreover, the intersection between the class of Lie algebras of type I and the class of completely solvable Lie algebras is exactly the class of nilpotent Lie algebras. We summarize some properties of Lie algebras of type I (see \cite{OnVi} for further details about Lie groups or Lie algebras of type I):

\begin{itemize}
	\item a Lie algebra of type I is in particular unimodular.
	\item any subalgebra of a Lie algebra of type I is of type I as well.
	\item let $\g$ be a Lie algebra with Levi decomposition $\g=\mathfrak r\oplus \mathfrak s$ where $\mathfrak r$ is a radical and $\mathfrak s$ is the Levi factor. According to \cite{AHG} we have that $\g$ is of type I if and only if $\mathfrak s$ is compact and $\mathfrak r$ is of type I. 
\end{itemize}
The study of Lie groups (Lie algebras) of type I terefore reduces to the solvable case. 

\

\section{LCS structures on solvable Lie algebras of type I}

In this section we study LCS structures on Lie algebras of type I and we show that all of them are of the first kind. 
In \cite{BM} it is proved that any LCS structure on a nilpotent Lie algebra is of the first kind and the Lee vector is central. Then using the classification of nilpotent Lie algebras of dimension $4$ and $6$ they determined which of them admit a LCS structure. In our work we extend the study of LCS structures to a larger class of Lie algebras which contains the nilpotent ones, namely the Lie algebras of type I.

\medskip

We begin our study with this result about the Morse-Novikov cohomology of certain class of Lie algebra, which includes the Lie algebras of type I. 

\begin{thm}\label{millo}
	Let $\g$ be a Lie algebra and let $\theta\in\g^*$ such that $\theta\neq0$ and $d\theta=0$. If there exists $A\in\g$ such that $\theta(A)=1$ and $\ad_A$ has all its eigenvalues in $i\R$, then $H^*_\theta(\g)=0$.
\end{thm}

\begin{proof}
Let $A\in\g$ such that $\theta(A)=1$ and $\ad_A$ has only imaginary eigenvalues (in $i\mathbb R$) and let
\[\ad_A^*: (\ker\theta)^* \to (\ker\theta)^*,\]
be the adjoint operator of $\ad_A:  \ker\theta \to \ker\theta$. This operator can be extended to $\alt^*(\ker\theta)^*$ by
\[\ad_A^*(\alpha\wedge\beta)=\ad_A^*\alpha\wedge\beta + \alpha\wedge\ad_A^*\beta,\]
for $\alpha\in(\ker\theta)^*$ and $\beta\in\alt^p(\ker\theta)^*$.
Since $\ad_A^*$ commutes with the exterior derivative $d$, it defines a map on cohomology 
\[\ad_A^*: H^k(\ker\theta) \to H^k(\ker\theta).\]
Let $\operatorname{Spec}^k_A$ be the set of eigenvalues of the operator $\ad_A^*: H^k(\ker\theta) \to H^k(\ker\theta)$.
In \cite{Mil} it is proved that the twisted cohomology 
$H_\theta^*(\g)$ is non trivial if and only if 
\[ 1\in \displaystyle{\bigcup_{k=1}^n} \operatorname{Spec}^k_A.\] 

The eigenvalues of $\ad_A^*$ in $\alt^k(\ker\theta)^*$ are sums of $k$ eigenvalues of $\ad_A^*$ in $(\ker\theta)^*$ (see \cite{Ar}). It can be seen that the eigenvalues of $\ad_A^*$ on $H^k(\ker\theta)$ are a subset of the previous ones, that is, they also are sums of $k$ eigenvalues of $\ad_A^*$, or equivalently, of $\ad_A$. Since $\ad_A$ has imaginary eigenvalues then $1\notin \operatorname{Spec}^k_A$ for any $k$, and therefore $H_\theta^*(\g)$ is trivial. 
\end{proof}

Recall that for a nilpotent Lie algebra there is a well know result due to Dixmier:  

\begin{thm}[\cite{Dix}]
	Let $\g$ be a nilpotent Lie algebra of dimension $n$. If $\theta\in\g^*$ is non-zero and $d\theta=0$, then $H^p_\theta(\g)=0$ for $0\leq p\leq n$.
\end{thm}	

Note now that this result is a direct consequence of Theorem \ref{millo}. Also, the class of Lie algebras of type I is a particular case where we can use Theorem \ref{millo} and we obtain the following result

\begin{cor}\label{tipoI-coho-trivial}
	Let $\g$ a Lie algebra of type I and let $\theta\in\g^*$ such that $\theta\neq0$ and $d\theta=0$. Then $H^*_\theta(\g)=0$.
\end{cor}

\smallskip

We consider now a Lie algebra $\g$ endowed with a LCS structure. As a consequence of Theorem \ref{millo} we obtain 
\begin{thm}\label{teo1}
	Let $(\omega,\theta)$ be a LCS structure on a unimodular Lie algebra $\g$. If there exists $A\in\g$ such that $\theta(A)=1$ and $\ad_A$ has all its eigenvalues in $i\R$, then $(\omega,\theta)$ is of the first kind.
\end{thm}

\begin{proof}
According to Theorem \ref{millo} we have that $H^*_\theta(\g)=0$.
In particular, $H_\theta^2(\g)=\{0\}$, then $\omega$ is $\theta$-exact. Since $\g$ is unimodular, it follows from Proposition \ref{1º_exacta} that the LCS structure $(\omega,\theta)$ is of the first kind.
\end{proof}

Note that the converse of Theorem \ref{teo1} is not true as we show in the following example. Let $\g$ be the $4$-dimensional Lie algebra denoted by $\mathfrak d_4$ with structure equation $(14,-24,-12,0)$ in the Salamon notation \footnote{The Salamon notation for the Lie algebra $\mathfrak d_4=(14,-24,-12,0)$ means that we fix a coframe $\{e^1,e^2,e^3,e^4\}$ for $\mathfrak d_4^*$ such that $de^1=e^{14}$, $de^2=-e^{24}$, $de^3=-e^{12}$ and $de^4=0$, where $e^{ij}$ means $e^i\wedge e^j$.}. According to \cite{ABP}
$$\omega=e^{12}-\sigma e^{24}, \quad \sigma>0, \quad \quad \theta=\sigma e^4$$
is a LCS structure of the first kind. But we can show that there does not exist an element $A\in\g$ satisfying the conditions of Theorem \ref{teo1}. 
Indeed, suppose $A=ae_1+be_2+ce_3+de_4$, with $a,b,c,d\in\R$. Since $\theta(A)=1$, then $d=\frac 1\sigma\neq0$. The operator $\ad_A$ can be written in the basis $\{e_1,e_2,e_3,e_4\}$ as
\[\begin{pmatrix}
d&0&0&-a\\
0&-d&0&b\\
-b&a&0&0\\
0&0&0&0
\end{pmatrix}.\]
Therefore the eigenvalues of $\ad_A$, $\{0,d,-d\}$ are not all imaginary, since $d\neq0$.

\bigskip

As a consequence of Theorem \ref{teo1} we obtain that a LCS structure on a Lie algebra of type I (if it exists) is necessarily of the first kind. 
%Recall that a Lie algebra $\g$ is said to be of type I if all the eigenvalues of the operator $\ad_X$ are imaginary for all $X\in\g$ (some of them may be equal to zero). This class of Lie algebras is in some sense opposite to the class of completely solvable Lie algebras (see \cite{OnVi} for further details about Lie groups or Lie algebras of type I). Note that any nilpotent Lie algebra is of type I.

\medskip

\begin{cor}\label{ImgPuras-1tipo}
If $\g$ is a Lie algebra of type I, then any LCS structure on $\g$ (if it exists) is of the first kind.
\end{cor}

\medskip
We now exhibit an example of a $6$-dimensional Lie algebra not of type I, which admits LCS structures only of the first kind. Therefore this example shows that the converse of Corollary \ref{ImgPuras-1tipo} is not true. 
\begin{ejemplo}
	Let us consider the Lie algebra $\g$ with Lie bracket given by
	\[
	[e_2,e_3]=e_1, \quad \quad
	[e_2,e_5]=e_2,\]
	\[ [e_3,e_5]=-e_3, \quad \quad 
	[e_4,e_5]=e_1, \]
	in the basis $\{e_1,e_2,e_3,e_4,e_5,e_6\}$. Clearly $\g$ is not of type I because $\ad_{e_5}$ has real eigenvalues. Let $\{e^1,e^2,e^3,e^4,e^5,e^6\}$ be the dual basis. Then the differential $d: \g^* \to \alt^2\g^*$ is given by
	$$de^1=-e^{23}-e^{45}, \quad de^2=-e^{25}, \quad de^3=e^{35}, \quad de^4=de^5=de^6=0.$$
	It can be shown that 
	$$\left\{\begin{array}{l}
	\omega= e^{16}-e^{23}-e^{45}\\
	\theta=e^6
	\end{array}\right. $$
	is a LCS structure on $\g$. Moreover it is of the first kind, since $\omega=d_{e^6}(e^1)$.
	
	We show next that the Lie algebra $\g$ admits LCS structures only of the first kind. In order to prove this fact, we consider a generic closed $1$-form $\theta$ on $\g$, that is, $\theta=ae^4+be^5+ce^6$ for $a,b,c\in \R$. We suppose that $\theta$ is the Lee form for a LCS form $\omega$.
	
	If $a^2+c^2\neq0$ we define $A=\frac{1}{a^2+c^2}(ae^4+ce^6)$, then $\theta(A)=1$ and zero is the only eigenvalue of $\ad_A$. It follows from Theorem \ref{teo1} that 
	%$H_\theta^*(\g)$ is trivial. In particular, $H_\theta^2(\g)=\{0\}$, then $\omega$ is $\theta$-exact. According to Proposition \ref{1º_exacta} we get that 
	if $(\omega,\theta)$ is a LCS structure then it must be of the first kind.
	
	If $\theta=be^5$ it can be seen that $\omega$ is degenerate, therefore there is no LCS structure in this case.
\end{ejemplo}

\begin{obs}
	According to \cite{AO2}, any solvable unimodular Lie algebra admitting a Vaisman structure is a Lie algebra of type I. Therefore the LCS structure underlying the Vaisman structure is of the first kind due to Corollary \ref{ImgPuras-1tipo}. 
	Moreover, any LCK structure of the first kind on a solvable unimodular Lie algebra is Vaisman according to \cite{S2}.
	%Moreover, any solvable unimodular Lie algebra admitting a LCK of the first kind is Vaisman according to \cite{S2}.\footnote{Me olvidé pq puse eso, esta bien?}
\end{obs}

\begin{obs}
	%As a consequence of Theorem \ref{millo} we have that any Lie algebra of type I admitting a LCS structure $(\omega, \theta)$ has trivial adapted cohomology, that is, $H_\theta^*(\g)=0.$ 
	From Corollary \ref{tipoI-coho-trivial} we know that the Morse-Novikov cohomology for Lie algebra of type I admitting a LCS structure $(\omega, \theta)$ is trivial.
	This fact was known for a Vaisman Lie algebra, since there is an injection $i^*:H^*_\theta(\g)\to H^*_\theta(\Gamma\backslash G)$ (see \cite{K}) and the latter is trivial acording to \cite{LLMP}.	 
\end{obs}

\begin{obs}\label{contact}
\Cref{ImgPuras-1tipo} generalizes the result in \cite{BM} which states that any LCS structure on a nilpotent Lie algebra is of the first kind, since every nilpotent Lie algebra is of type I. 
\end{obs}

\smallskip

We can see that Proposition \ref{lcs_contact} and Proposition \ref{lcs_symplectic} can be adapted for Lie algebras of type I. 

\begin{thm}\label{lcs_contact_I}
	There is a one to one correspondence between $(2n+1)$-dimensional contact Lie algebras of type I $(\h,\eta)$ endowed with a derivation $D$ such that $\eta\circ D=0$ and $D$ has only imaginary eigenvalues and $(2n+2)$-dimensional Lie algebras of type I endowed with a LCS structure of the first kind $(\g,\omega,\theta)$. The relation is given by $\h=\ker\theta$, $\omega=d_\theta \eta$ and $D=\ad_U$.
\end{thm}

\begin{proof}
	If $\h$ is a $(2n+1)$-dimensional contact Lie algebra of type I endowed with a derivation $D$ such that $D$ has only imaginary eigenvalues, then it follows from Proposition \ref{lcs_contact} that $\g=\R\ltimes_D\h$ carries a LCS structure of the first kind, and it is easy to see that $\g$ is a Lie algebra of type I.
	
	Conversely, if $\g$ is a Lie algebra of type I, then every subalgebra of $\g$ is of type I, in particular $\ker\theta$. Moreover $D=\ad_U$ has only imaginary eigenvalues. Now the result follows from Proposition \ref{lcs_contact}.
\end{proof}

\smallskip

\begin{thm}\label{lcs_symplectic_I}
	There is a one to one correspondence between $(2n+2)$-dimensional Lie algebras of type I admitting a LCS structure of the first kind with central Lee vector and symplectic Lie algebras of type I $(\mathfrak{s}, \beta)$ of dimension $2n$ endowed with a derivation $E$ such that $\beta(EX,Y)+\beta(X,EY)=0$ for all $X,Y\in\mathfrak s$ and $E$ has only imaginary eigenvalues.
\end{thm}

\begin{proof}
	According to Proposition \ref{lcs_symplectic} we only need to verify that the Lie algebra $\g$ is of type I if and only $\mathfrak s$ is of type I and the derivation $E$ has only imaginary eigenvalues. We can write $\g=\R U\oplus\R V\oplus\mathfrak s$ and the relation between the Lie bracket of $\g$ and the Lie bracket of $\mathfrak s$ is given by
	$[X,Y]=\beta(X,Y)V + [X,Y]_{\mathfrak s}$ and $[U,X]=EX$, for all $X,Y\in\mathfrak s$, where $\beta$ is the symplectic form on $\mathfrak s$. Let $\{e_1,\dots,e_n\}$ a basis of $\mathfrak s$, then we can write
	\[\ad_X^\g=\left(\begin{array}{c|c|ccc}
	0&0&&0&\\
	\hline
	0&0&&w&\\
	\hline
	&&&&\\
	z^t&0&&\ad_X^\mathfrak s\\
	&&&&
	\end{array}\right)\]
	in the basis $\{U,V,e_1,\dots,e_n\}$ of $\g$ , where $z=-EX$, $w=(w_1,\dots,w_n)$ satisfies $w_i=\beta(X,e_i)$ and $\ad_X^\mathfrak s$ represents the adjoint action of $X$ in $(\mathfrak s,[X,Y]_{\mathfrak s})$. It is clear that  $\text{Spec}(\ad_X^\mathfrak s)\cup\{0\}=\text{Spec}(\ad_X^\g)$. Therefore $\g$ is of type I if and only $\mathfrak s$ is of type I and the derivation $E$ has only imaginary eigenvalues.
\end{proof}

%\begin{proof}
%We now apply Proposition \ref{lcs_contact} to show that the correspondence between Lie algebras of dimension $2n$ admitting a LCS structure of the first kind and $(2n-1)$-dimensional contact Lie algebras endowed with a certain derivation. Then the Lie algebra $\g$ decomposes $\g=\R A\ltimes_D\h$ where $\h$ is a contact Lie algebra and $D$ is a contact derivation. With this description it is easy to see that $\h$ is also a Lie algebra of type I.	
%\end{proof}

\begin{obs}\label{symplectica&unimodularNoExacta}
	Note that the symplectic forms in the correspondence of Theorem \ref{lcs_symplectic_I} are not exact. Indeed, a unimodular Lie algebra does not admit exact symplectic forms (see \cite{Kru}).
\end{obs}

\begin{obs}
	Theorem \ref{lcs_symplectic_I} generalizes the result in \cite{BM} where a relation between $(2n+2)$-dimensional nilpotent Lie algebras admitting a LCS structure and nilpotent Lie algebras with a symplectic structure and a symplectic nilpotent derivation was established.
\end{obs}

\medskip

We show an application of Theorem \ref{lcs_symplectic_I}. Let us start with a particular case of symplectic Lie algebras of type I, namely the K\"ahler flat Lie algebras, that is, a flat Lie algebra $(\mathfrak s, \pint)$ with a complex structure $J$ such that $(\mathfrak s,J, \pint)$ is K\"ahler where the fundamental form $\omega$ is determined by $(J, \pint)$ (see \cite{BDF,Mi} for a nice description of flat Lie algebras). In particular it is proved in \cite{Mi} that a Lie algebra $(\mathfrak s, \pint)$ is flat if and only if $\mathfrak s$ splits as an orthogonal direct sum $\mathfrak b\oplus\mathfrak u$ where $\mathfrak b$ is an abelian subalgebra, $\mathfrak u$ is an abelian ideal and the linear map $\ad_X$ is skew-symmetric for all $X\in\mathfrak b$. This fact implies that any flat Lie algebra is a Lie algebra of type I.

We consider next $D$ a skew-symmetric derivation of $\mathfrak s$ which commutes with the complex structure, it means that $D$ is a symplectic derivation. Then it follows from Theorem \ref{lcs_symplectic_I} that the Lie algebra $\g$ which is the double extension of $(\mathfrak s,\omega)$ by the derivation $D$ admits a LCS structure of the first kind. 
Moreover according to \cite{AO2} this LCS structure arises from a Vaisman structure on $\g$.
% it can be shown that this LCS structure is actually a Vaisman structure. In particular we recover the result of \cite{AO2} where any Lie algebra with a Vaisman structure is obtained by a double extension of a K\"ahler flat Lie algebra and a skew-symmetric derivation.

\medskip

It is also possible to take a symplectic derivation of a K\"ahler flat Lie algebra $\mathfrak s$ which is not skew-symmetric, and using Theorem \ref{lcs_symplectic_I} we obtain a Lie algebra of type I with a LCS structure which does not arise from a Vaisman structure, according to \cite{AO2}. For example we start with the K\"ahler flat Lie algebra $\mathfrak s=\mathfrak r'_{3,0}\times\R$ with structure equation $(0,-13,12,0)$ and symplectic form $\omega=-e^{14}+e^{23}$. We consider the Lie algebra $\g$ given by the double extension of $\mathfrak s$ by $\omega$ and the following derivation 
$$D=\begin{pmatrix}
0&0&0&0\\
0&0&-b&0\\
0&b&0&0\\
a&0&0&0
\end{pmatrix},$$
for $a,b\in\R$ and $a\neq0$, in the basis $\{e_1,e_2,e_3,e_4\}$ of $\mathfrak s$. Then $\g=\R e_6\ltimes_D(\R e_5\oplus_\omega\mathfrak s)$ is a $6$-dimensional Lie algebra of type I admitting a LCS structure $\Omega=-e^{14}+e^{23}-e^{56}$ with Lee form $\theta=e^6$, which is not LCK.

\

%Let us recall now that for a the symplectic $2n$-dimensional Lie algebra $\mathfrak s$ with non trivial center, it can be obtained by a \textit{symplectic double extension} of a $(2n-2)$-dimensional symplectic Lie algebra endowed with a suitable derivation (see \cite{MR1} for more details about symplectic double extensions of symplectic Lie algebras).  

For a nilpotent Lie algebra with a LCS structure is shown in \cite[Theorem 5.15]{BM} that the Lee vector is a central vector and therefore they can be obtained as a double extension of a symplectic nilpotent Lie algebra with a suitable derivation. 
%Moreover, this symplectic nilpotent Lie algebra may obtained by a sequence of symplectic double extensions by nilpotent derivations from the abelian Lie algebra of dimension $2$. 
%\footnote{Para las tipo I vale que la extexsion simplectca doble de una tipo I es tipo I si la derivacion tiene autovalores imaginarios, esto extenderia lo que se sabe para las nilpotentes. Pero para agregar eso deberia hacer primero una intruduccion de lo que es la doble extension simplectica detallando bien los corchetes. No se si vale la pena pq despues no lo uso.}

For a Lie algebra of type I the Lee vector is not necessarily in the center of the Lie algebra. In the next subsection we exhibit an example of a non-nilpotent $4$-dimensional Lie algebra of type I admitting a LCS structure with non central Lee vector, which is then not covered by Proposition \ref{lcs_symplectic_I} and the LCS structure does not arise from a symplectic Lie algebra (see Example \ref{Ex.r'30xR} below).

%And we cannot apply the same argument that in the nilpotent case about symplectic double extension. Even though it is possible to see that any simplectic double extension of a Lie algebra of type I is a symplectic Lie algebra of type I if and only if the simplectic derivation has imaginary eigenvalues 

\medskip

\subsection{Dimension $4$}

In this section we determine all $4$-dimensional solvable Lie algebras of type I. 
For the sake of completeness, we exhibit the classification of LCS structures on $4$-dimensional solvable Lie algebras of type I.
For a complete classification of LCS structures on $4$-dimensional Lie algebras see \cite{ABP}.

Let us recall the notation of \cite{ABDO} for some Lie algebras in low dimension in the next table. Note that $\mathfrak r_{3,-1}$ is the Lie algebra $\mathfrak e(1,1)$ of the group of rigid motions of Minkowski $2$-space and $\mathfrak r'_{3,0}$ is the Lie algebra $\mathfrak e(2)$ of the group of rigid motions of Euclidean $2$-space.

\begin{table}[H]

	\def\arraystretch{1.4}
	
%	{\resizebox{1 \textwidth}{!}{
			\begin{tabular}{l|l|l}
				%\toprule
				Lie algebra &  Lie bracket & Salamon notation \\
				\toprule 
				%\endfirsthead
				\hline
				
				$\mathfrak {aff}(\R)$ & $[e_1, e_2] = e_2$ & (0,-12)\\
				\hline
				
				$\h_3$ & $[e_1, e_2] = e_3$ & (0,0,-12)\\
				\hline
				
				$\mathfrak r_{3,-1}$ &  $[e_1, e_2] = e_2$, $[e_1, e_3] = -e_3$ & (0,-12,-13)\\			
				\hline
				
				$\mathfrak r'_{3,0}$ &  $[e_1, e_2] = -e_3$, $[e_1, e_3] = e_2$ & (0,-13,12)\\
				
				\hline
				
				$\mathfrak n_4$ & $[e_1, e_4] = -e_2$,	$[e_2,e_4]=-e_3$ & (0,14,24,0)\\
				
				\hline
				
				$\mathfrak d'_{4,0}$ &  $[e_1, e_2] = e_3$, $[e_1, e_4] = e_2$,	$[e_2,e_4]=-e_1$ & (24,-14,-12,0)\\
				
				\hline					
				
				\bottomrule
			\end{tabular}
%	}}
	%\caption{$5$-dimensional Lie algebras of type I with a contact structure.}
	%\label{contacto_typeI}
\end{table}

\begin{prop}\label{4-LieAlgI}
	If $\g$ is a $4$-dimensional solvable Lie algebra of type I then $\g$ is isomorphic to one of the following: $\mathfrak r'_{3,0}\times\R$, $\mathfrak h_3\times\R$, $\mathfrak{d}'_{4,0}$ or $\mathfrak n_4$.
\end{prop} 

\begin{proof}
	According to \cite{ABDO} any $4$-dimensional solvable real Lie algebra $\g$ is a semidirect product of $\R$ with a $3$-dimensional unimodular ideal. Then $\g$ can be written as $\g=\R\ltimes_D\u$ where $\u$ is an ideal of $\g$ isomorphic to either $\R^3$, $\h_3$, $\mathfrak r_{3,-1}$ or $\mathfrak r'_{3,0}$ and $D$ is a derivation of $\u$.
	
	If $\u=\h_3$ then according to \cite{ABDO} a generic derivation of $\u$ has the form
    $$D=\left(\begin{array}{cc|c}
	a&c&\\
	b&d&\\
	\hline
	x&y&0 
    \end{array}\right),$$ 
    with respect to the basis $\{e_1,e_2,e_3\}$ such that $[e_1,e_2] = e_3$.
    We may assume $x=y=0$ after a change of basis.
    It is easy to see that $\g$ is of type I if and only if $\det(A)\geq0$ and $\tr(A)=0$ where 
    $A=\begin{pmatrix}
    a&c\\
    b&d	
    \end{pmatrix}$. We can assume that $A$ is in its Jordan form, then $A$ takes the form $A=0$ or  
    $A=\begin{pmatrix}
    0&1\\
    0&0	
    \end{pmatrix}$ or 
    $A=\begin{pmatrix}
    0&-1\\
    1&0	
    \end{pmatrix}$.
    Therefore $\g$ is isomorphic to either $\mathfrak h_3\times\R$ or $\mathfrak n_4$ or $\mathfrak{d}'_{4,0}$.

    If $\u=\mathfrak r_{3,-1}$ then $\g$ is not of type I.
    
    If $\u=\mathfrak r'_{3,0}$ then according to \cite{ABDO} a generic derivation of $\u$ has the form
    $$D=\left(\begin{array}{ccc}
    0&0&0\\
    c&a&-b\\
    d&b&a 
    \end{array}\right),$$ 
    with respect to the basis $\{e_1,e_2,e_3\}$ such that $[e_1,e_2] = e_3$ and $[e_1,e_3] = -e_2$. It is easy to see that in this case $\g$ is of type I if and only if $a=0$, it may be shown that $\g=\mathfrak r'_{3,0}\times\R$ according to \cite{ABDO}.
    
    If $\u=\R^3$, according to \cite{ABDO} we may assume that $D$ is one of the following: 
    $$D=\left(\begin{array}{ccc}
    a&0&0\\
    0&b&0\\
    0&0&c 
    \end{array}\right), \quad 
   D=\left(\begin{array}{ccc}
    	a&0&0\\
    	0&b&1\\
    	0&0&b 
    \end{array}\right), \quad
    D=\left(\begin{array}{ccc}
    	a&1&0\\
    	0&a&1\\
    	0&0&a 
    \end{array}\right), \quad
    D=\left(\begin{array}{ccc}
    	a&0&0\\
    	0&b&-1\\
    	0&1&b 
    \end{array}\right).$$
    It can be shown that $\g=\R\ltimes_D \R^3$ is of type I if and only if either $D=0$ or
  $$
  D=\left(\begin{array}{ccc}
  0&0&0\\
  0&0&1\\
  0&0&0 
  \end{array}\right), \quad
  D=\left(\begin{array}{ccc}
  0&1&0\\
  0&0&1\\
  0&0&0 
  \end{array}\right), \quad
  D=\left(\begin{array}{ccc}
  0&0&0\\
  0&0&-1\\
  0&1&0 
  \end{array}\right).$$
  Therefore (except for the abelian Lie algebra) $\g$ is isomorphic to $\mathfrak h_3\times\R$, or $\mathfrak n_4$, or $\mathfrak r'_{3,0}\times\R$.
 \end{proof}

We exhibit now LCS structures on the non nilpotent Lie algebras obtained in Proposition \ref{4-LieAlgI}, namely $\mathfrak r'_{3,0}\times\R$ and $\mathfrak{d}'_{4,0}$. The nilpotent ones were studied in \cite{BM}.

\begin{ejemplo}\label{Ex.r'30xR}
	We consider first the $4$-dimensional Lie algebra $\mathfrak r'_{3,0}\times\R$ with structure equations $(0,-13,12,0)$. Then $\mathfrak r'_{3,0}\times\R$ admits a basis $\{e_1,e_2,e_3,e_4\}$ such that the only no zero Lie bracket are $[e_1,e_2]=e_3$ and $[e_1,e_3]=-e_2$. The adjoint action of $e_1$ is 
$\ad_{e_1}=\begin{pmatrix}
0&-1\\
1&0
\end{pmatrix}$, 
therefore it is a Lie algebra of type I, but not nilpotent. 
It was shown in \cite{AO1} that the Lie algebra $\mathfrak r'_{3,0}\times\R$ admits LCS structures only of the first kind. This fact can be seen as a direct consequence of \Cref{ImgPuras-1tipo}. It follows from \cite{ABP} that the generic LCS structure for this Lie algebra is 
\[\omega=e^{13} - e^{24}, \quad \theta=e_4,\]
where $\{e^1,e^2,e^3,e^4\}$ is the dual basis of $\{e_1,e_2,e_3,e_4\}$.
Note that the Lee vector for this LCS structure is $e_1$ and it is not a central vector, therefore this example is not covered by Theorem \ref{lcs_symplectic_I}. It is known that the simply connected Lie group associated to this Lie algebra admits lattices, therefore the corresponding solvmanifold has a LCS structure of the first kind. This solvmanifold admits a complex structure and as a compact complex surface it is a Inoue surface of type $S^0$.
This example was our motivation to study Lie algebras of type I with LCS structures.
\end{ejemplo}

\medskip

\begin{ejemplo}
%We will see next another examples of a $4$-dimensional Lie algebra satisfying the conditions of \Cref{ImgPuras-1tipo}. Let $\g=\R e_0\ltimes\h_3$ be the unimodular Lie algebra given by $[e_0,e_1]=ae_1+be_2+xe_3$, $[e_0,e_2]=ce_1-ae_2+ye_3$, $[e_1,e_2]=e_3$ and $e_3\in\z(\g)$, in a basis $\{e_0,e_1,e_2,e_3\}$ of $\g$, with $a,b,c,x,y\in\R$. It is easy to see that $\g$ is a Lie algebra of type I if and only if $a^2+bc\leq0$.
%If $a^2+bc=0$, then $\g$ is nilpotent and $\g$ can be isomorphic to either $\R\times\h_3$ or $\mathfrak n_4$ and they are been studied in \cite{BM}. We focus on the non-nilpotent ones, that is, $a^2+bc<0$. In this case, any of these Lie algebras is isomorphic to the following Lie algebra: 
%\[[e_0,e_1]=-e_2, \quad [e_0,e_2]=e_1, \quad [e_1,e_2]=e_3, \quad e_3\in\z(\g).\]
%This Lie algebra corresponds to the solvable non-nilpotent Lie algebra $\mathfrak{d}'_{4,0}$ in the terminology of \cite{ABDO}.

We consider now the Lie algebra denoted by $\mathfrak{d}'_{4,0}$ with structure equations $(24,-14,-12,0)$. Then $\mathfrak{d}'_{4,0}$ admits a basis $\{e_1,e_2,e_3,e_4\}$ with Lie bracket given by $[e_1,e_2]=e_3$, $[e_1,e_4]=e_2$ and $[e_2,e_4]=-e_1$.
We determine next the LCS forms in $\mathfrak{d}'_{4,0}$.
It is easy to see that a LCS structure $(\omega, \theta)$ in $\mathfrak{d}'_{4,0}$ is given by
\[\omega=a_1e^{14}+ a_2e^{24} + a_3e^{34} + a_3e^{12}, \quad \theta=e_4,\]
where $\{e^1,e^2,e^3,e^4\}$ is the dual basis of $\{e_1,e_2,e_3,e_4\}$, for some $a_1,a_2,a_3\in\R$ and $a_3\neq0$. It follows from \cite{ABP} that these LCS structures on $\mathfrak{d}'_{4,0}$ are equivalent, up to automorphisms of the Lie algebra, to one of the following:
\[\omega= e^{12} - ce^{34}, \quad \theta=ce^4, \quad\quad c>0.\]
It follows from \Cref{ImgPuras-1tipo} that all these LCS structures are of the first kind. It is also easy to verify this fact from the definition. It is well known that the corresponding simply connected Lie group $G$ associated to the Lie algebra $\mathfrak{d}'_{4,0}$ admits lattices. 
In fact, in \cite{AO2} three families of lattices in $G$, namely $\Lambda_{k,\frac{\pi}{2}}, \Lambda_{k,\pi}$ and $\Lambda_{k,2\pi}$ with $k\in\mathbb N$, were exhibited. It was also shown that the solvmanifolds $\Lambda_{k,\frac{\pi}{2}}\backslash G$ and $\Lambda_{k,\pi}\backslash G$ are secondary Kodaira surfaces, while $\Lambda_{k,2\pi}\backslash G$ is diffeomorphic to a nilmanifold $\Gamma\backslash \R\times H_3$, which is a primary Kodaira surface.
It can be seen that the LCS $2$-form induced by $\omega$ on $\Gamma\backslash \R\times H_3$ is not invariant for the nilpotent group $\R\times H_3$.

%The most of the solvmanifolds obtained as a quotient correspond to a secondary Kodaira surface with some exception where a primary Kodaira surface is obtained (see \cite{Has}).%\footnote{REVISAR: la LCS de este ejemplo pasa a la primary kodaira surface, pero es invariante por el grupo de Lie nilpotente? Estaria bueno ver si la variedad de Kodaira-Thurston admite LCS diferentes, unas invariantes por un nilpotente pero no invariantes por el soluble eso sería algo nuevo para decir, que no aparece en el trabajo de los italianos}
\end{ejemplo}

%\begin{obs}
%It can be seen that $\mathfrak r'_{3,0}\times\R$, $\mathfrak h_3\times\R$, $\mathfrak{d}'_{4,0}$ and $\mathfrak n_4$ are the only examples of $4$-dimensional solvable Lie algebras of type I. All of them admit LCS structures (see \cite{ABP}), and therefore according to Corollary \ref{ImgPuras-1tipo} these LCS structures are of the first kind.
%\end{obs}   

We summarize the Lie algebras of type I in dimension $4$ and we exhibit a LCS structure for each of them in the next table.

	%			\multirow{2}{*}{$\R\times\h_3$} & \multirow{2}{*}{$(0,0,-12,0)$} & $\omega=e^{12}-e^{34}$ & \multirow{2}{*}{$\checkmark$} \\
	%			&&$\theta=e^4$&
	%			\hline

	%			\multirow{2}{*}{$\mathfrak d'_{4,0}$} &  	\multirow{2}{*}{$(24,-14,-12,0)$} & $\theta= \sigma e^4$  & $\delta>0$  \\
	%			&&$\omega=\pm(e^{12}-(\delta+\sigma)e^{34})$ & $\sigma\neq -\delta, 0$ \\
	%			\hline					

\begin{table}[H]
	\def\arraystretch{1.4}
	\centering
	%{\resizebox{\textwidth}{!}{
			\begin{tabular}{l|l|l|c}
				%\toprule
				Lie algebra &  structure equations & LCS structure & nilpotent \\
				\toprule 
				%\endfirsthead
				\hline
				
                $\h_3\times\R$ & $(0,0,-12,0)$ & $\omega=e^{12}-e^{34}, \quad \theta=e^4$ &$\checkmark$ \\
                \hline
                
				$\mathfrak n_4$ &  	$(0,14,24,0)$ & $\omega=e^{13}-e^{24}, \quad \theta= e^1$  & $\checkmark $  \\
			
				\hline
				
				$\mathfrak r'_{3,0}\times\R$ & $(0,-13,12,0)$ & $\omega=e^{12}+e^{13}-e^{24}, \quad \theta= e^4$ & $\times$ \\
				
				\hline
				
				$\mathfrak d'_{4,0}$ & $(24,-14,-12,0)$ & $\omega=e^{12}-e^{34}, \quad\theta= e^4$  & $\times$ \\

				\hline					
	
				\bottomrule
			\end{tabular}
%	}}
	\caption{$4$-dimensional Lie algebras of type I.}
	%\label{extensibles}
\end{table}

\begin{obs}
In dimension $4$ we have a converse of Corollary \ref{ImgPuras-1tipo}, that is, any solvable Lie algebra admitting LCS structures only of the first kind are Lie algebras of type I. The proof of this fact follows directly from the classification of LCS structures in $4$-dimensional Lie algebras given in \cite{ABP}. %\footnote{Si saco soluble vale? capaz se pueden agregar las reductivas}
\end{obs}	

\

\subsection{LCS Lie algebras of type I in dimension $6$}

Recall that $6$-dimensional nilpotent Lie algebras admitting a LCS structure were classified in \cite{BM}. In this section we focus on Lie algebras of type I, which are not nilpotent, admitting a LCS structure. 
As we mentioned in Theorem \ref{lcs_contact_I} there is a one to one correspondence between $2n$-dimensional Lie algebras of type I with a LCS structure and $(2n-1)$-dimensional contact Lie algebras of type I endowed with a suitable derivation.

First we study $5$-dimensional Lie algebras of type I with a contact structure. We recall first that $5$-dimensional Lie algebras with a contact structure were classified in \cite{Di}. We summarize this classification:
\begin{enumerate}[(i)]
	\item If $\g$ is solvable and non-decomposable there are 24 non-isomorphic $5$-dimensional Lie algebras admitting a contact structure (see \cite[Section $5$]{Di}). After some computations one can verify that the only solvable non-decomposable $5$-dimensional Lie algebras of type I admitting a contact $1$-form $\eta$ are: 
	\begin{enumerate}[(1)]
		\item  $[e_2, e_4] = e_1$, $[e_3, e_5] = e_1$, $\eta= e^1$
		\item  $[e_3, e_4] = e_1$, $[e_2, e_5] = e_1$, $[e_3, e_5] = e_2$, $\eta= e^1$
		\item  $[e_3, e_4] = e_1$, $[e_2, e_5] = e_1$, $[e_3, e_5] = e_2$, $[e_4, e_5] = e_3$, $\eta= e^1$
	    \item  $[e_2, e_3] = e_1$, $[e_2, e_5] = e_3$, $[e_3, e_5]=−e_2$, $[e_4, e_5] = e_1 $, $\eta= e^1$.
    \end{enumerate}

\bigskip

   \item If $\g$ is solvable and decomposable we have two options:
     \begin{enumerate}[(a)]
	   \item $\g=\mathfrak{aff}(\R)\oplus\h$ with $\h$ a $3$-dimensional Lie algebra different from $\R\ltimes_{\I} \R^2$,
	   \item $\g=\R\oplus\h$ where $\h$ is a $4$-dimensional Lie algebra with a exact symplectic form.
     \end{enumerate}
    In the first case $\g=\mathfrak{aff}(\R)\oplus\h$ is not of the type I because $\mathfrak{aff}(\R)$ is not of type I. 
    In the second case, $\g=\R\oplus\h$ cannot be unimodular according to Remark \ref{symplectica&unimodularNoExacta}, in particular it cannot be of type I.
    %Then $\g$ is of the type I if and only if $\h$ is of type I. According to \cite{Ov} the $4$-dimensional Lie algebra with a symplectic exact form are: $\mathfrak r_2\mathfrak r_2$, $\mathfrak r'_2$, $\mathfrak d_{4,\lambda}$, $\mathfrak d'_{4,\delta}$ with $\delta\neq0$ and $\mathfrak h_4$. It can be seen by direct computation that any of them is not a Lie algebra of the type I. 
 
\bigskip

\item If $\g$ is non-solvable, then $\g$ is one of the following: $\mathfrak{aff}(\R)\oplus\mathfrak{sl}(2)$, $\mathfrak{aff}(\R)\oplus\mathfrak{so}(3)$, or a semidirect product $\mathfrak{sl}(2)\ltimes\R^2$.
% with Lie bracket given by  $[X, e_2] = e_1$, $[Y , e_1] = e_2$, $[H, e_1] = e_1$, $[H, e_2]=−e_2$, $[X, Y ] = H$, $[H, X] = 2X$, $[H, Y ]=−2Y$ in the basis $\{ X, Y , H, e_1, e_2\}$.
But it is clear that $\mathfrak{aff}$ and $\mathfrak{sl}(2)$ are not of type I, therefore $\g$ is not of type I.
%$\mathfrak{aff}(\R)\oplus\mathfrak{sl}(2)$, $\mathfrak{aff}(\R)\oplus\mathfrak{so}(3)$ and $\mathfrak{sl}(2)\ltimes\R^2$ are not of type I.
Therefore we obtain the following result. 
\end{enumerate}

\begin{cor}
	If $\g$ is a $5$-dimensional Lie algebra of type I admitting a contact structure, then $\g$ is solvable and non-decomposable and it is isomorphic to one of the Lie algebras in Table \ref{contacto_typeI}.
\end{cor}

%In next table we summarize all the $5$-dimension Lie algebra of type I admitting a contact structure.

\begin{table}[H]
	\def\arraystretch{1.4}
	\centering
	{\resizebox{\textwidth}{!}{
			\begin{tabular}{l|l|l|c}
				%\toprule
				Lie algebra &  Lie bracket & contact structure & nilpotent \\
				\toprule 
				%\endfirsthead
				\hline
				
				$\h_5$ & $[e_2, e_4] = e_1$, $[e_3, e_5] = e_1$ & $\eta= e^1$  &$\checkmark$ \\
				\hline
				
				$\mathfrak n_1$ &  $[e_3, e_4] = e_1$, $[e_2, e_5] = e_1$, $[e_3, e_5] = e_2$ & $\eta= e^1$  & $\checkmark $  \\
				
				\hline
				
				$\mathfrak n_2$ & $[e_3, e_4] = e_1$, $[e_2, e_5] = e_1$, $[e_3, e_5] = e_2$, $[e_4, e_5] = e_3$ & $\eta= e^1$ & $\checkmark$ \\
				
				\hline
				
				$\mathfrak h$ &  $[e_2,e_3]=e_1$, $[e_2,e_5]=e_3$,	$[e_3,e_5]=-e_2$, $[e_4,e_5]=e_1$ & $\eta= e^1$  & $\times$ \\
				
				\hline					
				
				\bottomrule
			\end{tabular}
	}}
	\caption{$5$-dimensional Lie algebras of type I with a contact structure.}
	\label{contacto_typeI}
\end{table}

\begin{obs}
	For each Lie algebra in Table \ref{contacto_typeI}, it is possible to show that $\eta=e^1$ is the only contact form, up to automorphism of Lie algebras.
\end{obs}

The Lie algebras $\h_5$, $\mathfrak n_1$ and $\mathfrak n_2$ in Table \ref{contacto_typeI} are nilpotent, where $\h_5$ is the $5$-dimensional Heisenberg Lie algebra. The Lie algebra $\h$ from Table \ref{contacto_typeI} is the only $5$-dimensional unimodular solvable non nilpotent Lie algebra admitting a Sasakian structure according to \cite{AFV}. According to \cite{AO2} any $(2n-1)$-dimensional unimodular solvable Lie algebra $\mathfrak s$ admitting a Sasakian structure has non trivial center generated by the Reeb vector.
Moreover, this Lie algebra $\mathfrak s$ is a subalgebra of a Vaisman Lie algebra $\g$ of dimension $2n$, and according to \cite{AO2} this Lie algebra $\g$ is of type I, and therefore $\mathfrak s$ is a Lie algebra of type I as well. 

\begin{cor}
	Any unimodular solvable Lie algebra admitting a Sasakian structure is a Lie algebra of type I.
\end{cor}

\

We next determine all the derivations for each Lie algebra in Table \ref{contacto_typeI}. Computations have been performed with the help of Maple. Any derivation $D$ of each Lie algebra can be written in the basis $\{e_1,e_2,e_3,e_4,e_5\}$ as follows

$$\h_5: D=\begin{pmatrix}
d_{11}&d_{12}&d_{13}&d_{14}&d_{15}\\
0&d_{22}&d_{23}&d_{24}&d_{25}\\
0&d_{32}&d_{33}&d_{25}&d_{35}\\
0&d_{42}&d_{43}&d_{11}-d_{22}&-d_{32}\\
0&d_{43}&d_{53}&-d_{23}&d_{11}-d_{33}
\end{pmatrix},$$
  
$$\mathfrak n_1:D=\begin{pmatrix}
d_{11}&d_{12}&d_{13}&d_{14}&d_{15}\\
0&d_{22}&d_{23}&d_{24}&d_{25}\\
0&0&2d_{22}-d_{11}&0&d_{24}\\
0&0&d_{43}&2d_{11}-2d_{22}&d_{12}-d_{23}\\
0&0&0&0&d_{11}-d_{22}
\end{pmatrix},$$

$$\mathfrak n_2: D=\begin{pmatrix}
5d_{55}&d_{12}&d_{13}&d_{14}&d_{15}\\
0&4d_{55}&d_{23}&d_{24}&d_{25}\\
0&0&3d_{55}&d_{23}&d_{24}-d_{13}\\
0&0&0&2d_{55}&d_{12}-d_{23}\\
0&0&0&0&d_{55}
\end{pmatrix},$$
  
$$\h:D=\begin{pmatrix}
2d_{22}&d_{12}&d_{13}&d_{14}&d_{15}\\
0&d_{22}&d_{23}&0&d_{12}\\
0&-d_{23}&d_{22}&0&d_{13}\\
0&0&0&2d_{22}&d_{45}\\
0&0&0&0&0
\end{pmatrix},$$
where $d_{ij}\in\R$ for $1\leq i,j\leq 5$.

We use the derivations above to build explicit examples of Lie algebras of type I in dimension $6$ equipped with a LCS structure.

\subsubsection{Case $\n_1$}
We consider the nilpotent Lie algebra $\n_1$ with a general derivation $D$. It can be seen that the eigenvalues of $D$ are $\{d_{11}, d_{22},-d_{11}+2d_{22}, 2d_{11}-2d_{22}, d_{11}-d_{22}\}$. According to Theorem \ref{lcs_contact_I} the derivation $D$ should have only imaginary eigenvalues. Then $d_{11}=d_{22}=0$ and $D$ is a nilpotent matrix. Therefore the Lie algebra $\g=\R\ltimes_D\n_1$ will be a nilpotent Lie algebra (see \cite{BM} for a classification of LCS nilpotent Lie algebras in dimension $6$). 

\subsubsection{Case $\n_2$}
Next we regard the nilpotent Lie algebra $\n_2$ with a general derivation $D$. If we request $D$ to have only imaginary eigenvalues, then $D$ will be a nilpotent matrix and therefore $\g=\R\ltimes_D\n_2$ will be a nilpotent Lie algebra.

\subsubsection{Case $\h_5$}
We consider now the Heisenberg Lie algebra $\h_5$ with the contact structure given by $\eta=e^1$. According to Theorem \ref{lcs_contact_I} we request that $\eta\circ D=0$, then we obtain that $d_{11}=d_{12}=d_{13}=d_{14}=d_{15}=0$ and therefore $D$ can be written as follow
\begin{equation}\label{Der_h5}
D=\left(\begin{array}{c|c}
0&\\
\hline
& \begin{array}{c|c}
   A&B\\
   \hline
   C&-A^t
  \end{array}
\end{array}\right)
\end{equation}
where $A,B,C\in\mathfrak{gl}(2,\R)$ with $B^t=B$ and $C^t=C$. Note that in this case we have many possibilities to take $D$ satisfying that its eigenvalues are imaginary and $D$ is non-nilpotent. For example $B=C=0$, $\tr A=0$ and $\det A>0$.

Now we exhibit a example of a $6$-dimensional Lie algebra of type I admitting a LCS structure starting with the contact Lie algebra $(\h_5,\eta=e^1)$. This example is interesting because later we will show that the simply connected Lie groups associated to some of these Lie algebras admit lattices as we will show in the next section. %\footnote{Liquidar este caso}

We define $\g=\R\ltimes_D\h_5$ where
\begin{equation}\label{latt1}
D=\begin{pmatrix}
0&0&0&0&0\\
0&0&0&-1&0\\
0&0&0&0&0\\
0&1&0&0&0\\
0&0&b&0&0
\end{pmatrix}
\end{equation}
in the basis $\{e_1,e_2,e_3,e_4,e_5\}$ of $\h_5$.
Note that $D$ is as in \eqref{Der_h5} with $A=0$,
$C=\begin{pmatrix}
1&0\\
0&b
\end{pmatrix}$, and 
$B=\begin{pmatrix}
-1&0\\
0&0
\end{pmatrix}$.
Then $D$ is a derivation of $\h_5$, and $\g$ is a Lie algebra with Lie bracket 
\[
[e_2,e_4]=e_1, \quad
[e_3,e_5]=e_1,\]
\[[e_0,e_2]=e_4, \quad
[e_0,e_3]=be_5, \quad
[e_0,e_4]=-e_2,\]
for $b \in \mathbb{R}$ in the basis $\{e_0,e_1,e_2,e_3,e_4,e_5\}$ of $\g=\R e_0\ltimes_D\h_5$. Since $D$ is non-nilpotent and the eigenvalues of $D$ are $\{0,i,-i\}$, then according to Theorem \ref{lcs_contact_I} $\g$ is a non-nilpotent Lie algebra of type I and  
$$\left\{\begin{array}{l}
\theta=e^0\\
\omega=d_{\theta}e^1=-e^{01}-e^{24}- e^{35}, 
\end{array}\right. $$
%$\theta=e^0$ and 
%$$\omega=d_{\theta}e^1=-e^{01}-e^{24}- e^{35}$$ 
is a LCS structure of the first kind on $\g$ where $\{e^0,e^1,e^2,e^3,e^4,e^5\}$ denotes the dual basis of $\g^*$.  We denote this Lie algebra by $\g_{1,b}$. When $b=0$ the LCS structure on $\g_{1,0}$ arises from a Vaisman structure (see \cite{AO2}).
It can be proved that $\g_{1,b}$ is isomorphic to $\g_{1,1}$ for $b\neq 0$, and $\g_{1,1}$ is not isomorphic to $\g_{1,0}$.

\subsubsection{Case $\h$}
Finally we consider the Lie algebra $\h$ with the contact structure given by $\eta=e^1$. Assuming $\eta\circ D=0$ we reduce $D$ as follow 
\begin{equation}\label{Der_h}
D=\begin{pmatrix}
0&0&0&0&0\\
0&0&-d_{23}&0&0\\
0&d_{23}&0&0&0\\
0&0&0&0&d_{45}\\
0&0&0&0&0
\end{pmatrix}
\end{equation}
%$D$ is the derivation given by \eqref{Der_h}. 
Note that if $d_{23}\neq 0$ we can assume $d_{23}=1$.
%It can be shown that if change the basis $\{e_1,e_2,e_3,e_4,e_5\}$ to $\{e_5,e_3,e_4,e_2,e_1\}$ the matrix $D$ in the reordered basis has the same form as \eqref{Der_h5}.

Now we construct a example of a $6$-dimensional Lie algebra of type I admitting a LCS structure starting with the contact Lie algebra $(\h,\eta)$ with $\eta=e^1$. %Therefore, according to \Cref{ImgPuras-1tipo} the LCS structure is of the first kind.
In the next section we show that the simply connected Lie group associated to some of these Lie algebras admits lattices.

Let $\g$ be the Lie algebra given by $\g=\R\ltimes_D\h$ where $D$ is the following derivation of $\h$
\begin{equation}\label{latt2}
D=\begin{pmatrix}
	0&0&0&0&0\\
		0&0&-1&0&0\\
	0&1&0&0&0\\
		0&0&0&0&b\\
	0&0&0&0&0
\end{pmatrix},
\end{equation}
for some $b\in\R$.
Then the Lie bracket on $\g$ can be written in the basis $\{e_0,e_1,e_2,e_3,e_4,e_5\}$ as follows
	\[
	[e_2,e_3]=e_1, 
	[e_5,e_2]=-e_3,
	[e_5,e_3]=e_2,
	[e_5,e_4]=-e_1,\]
	\[[e_0,e_5]=be_4, 
	[e_0,e_2]=e_3,
	[e_0,e_3]=-e_2. \]
It is easy to verify that the eigenvalues of $D$ are $\{0,i,-i\}$. Since $\h$ is a contact Lie algebra of type I, it follows from Theorem \ref{lcs_contact_I} that $\g$ is a non-nilpotent Lie algebra of type I, and 
$$\left\{\begin{array}{l}
\theta=e^0\\
\omega=\omega=d_{e^0}(e^1)=-e^{01}-e^{23}- e^{45}, 
\end{array}\right. $$
is a LCS structure of the first kind on $\g$ where $\{e^0,e^1,e^2,e^3,e^4,e^5\}$ denotes the dual basis of $\g^*$. We denote this Lie algebra by $\g_{2,b}$. 
When $b=0$ the LCS structure on $\g_{2,0}$ arises from a Vaisman structure (see \cite{AO2}).
%\footnote{Se pueden dar todas las lcs y ver que son de first kind? se puede ver q $[\omega]_\theta=0$? No lo se todavia}
It can be proved that $\g_{2,b}$ is isomorphic to $\g_{2,1}$ for $b\neq 0$, and $\g_{2,1}$ is not isomorphic to $\g_{2,0}$.

\

\section{Solvmanifolds with LCS structures}

We now exhibit lattices in the simply connected Lie groups associated to the Lie algebras $\g_{1,b}$ and $\g_{2,b}$ constructed in the previous section.
 
\

We consider first the Lie algebra $\g_{1,b}=\R e_0\ltimes_D\h_5$ with $D$ given by \eqref{latt1}. We denote by $G_{1,b}$ the simply connected Lie groups associated to $\g_{1,b}$. Then $G_{1,b}=\R \ltimes_\varphi H_5$ is an almost nilpotent Lie group where 
$\varphi: H_5 \to H_5$ is
$$\varphi(t)=e^{tD}=\begin{pmatrix}
1 & 0 & 0 & 0 & 0 \\
0 & \cos t & 0 & -\sin t & 0 \\
0 & 0 & 1 & 0 & 0 \\
0 & \sin t & 0 & \cos t & 0 \\
0 & 0 & tb & 0 & 1 
\end{pmatrix},$$
and $H_5$ denotes the $5$-dimensional Heisenberg Lie group, i.e. the Euclidean manifold $\R^{5}$ equipped with the following product:
\[(z,x_1,y_1,x_2,y_2)\cdot (z',x'_1,y'_1,x'_2,y'_2)=(z+z'+\frac 12 (x_1y'_1-x'_1y_1+x_2y'_2-x'_2y_2),x_1+x'_1,y_2+y'_2).\]
 
We exhibit a lattice in $G_{1,b}$ for some values of $b$. Firstly, for any $k\in\N$ consider the lattice $\Gamma_k$ in $H_5$ given by $\Gamma_k= \frac{1}{2k}\Z\times\Z\times\Z\times\Z\times\Z$.
Let $\gamma=(\frac{m}{2k},p,q,r,s)\in\Gamma_k$, then $\varphi(t)(\gamma)\in\Gamma_k$ if and only if 
$$\left\{\begin{array}{l}
	p\cos t-r\sin t \in \Z\\
	p\sin t-r\cos t \in \Z\\
	qtb+s\in\Z. 
\end{array}\right. $$
Taking $t_0\in\{\frac{\pi }{2},\pi,2\pi\}$ and $b=\frac{1}{t_0}$ it is easy to see that $\Gamma_k$ is invariant under the subgroup generated by $\varphi(t_0)$. Therefore $\Lambda_{k,t_0}= t_0\Z\ltimes \Gamma_k$ is a lattice in $G_{1,\frac{1}{t_0}}$. Then $$\Lambda_{k,t_0}\backslash G_{1,\frac{1}{t_0}},$$ for any $k\in\Z$, are examples of $6$-dimensional solvmanifolds of type I admitting a LCS structure.
Note that for $t_0=2\pi$, then $\varphi(2\pi) = \I$, and $\Lambda_{k,2\pi}= 2\pi\Z\ltimes \Gamma_k$, which is isomorphic to a lattice in $\R \times H_{2n+1}$. Therefore, we have that the solvmanifold $\Lambda_{k,2\pi}\backslash G_{1,1/2\pi}$ is isomorphic to the nilmanifold $S^1 \times \Gamma_k\backslash H_{2n+1}$. For the other values of $t_0$ we obtain solvmanifolds which are covered for this nilmanifold.

\

We consider now, the Lie algebra $\g_{2,b}$. It is easy to see that this Lie algebra can be written as
\[\g = \R e_0 \ltimes_D (\R e_4\ltimes_\psi (\R e_5\ltimes_\rho \h_3)), \]
where \[D=\ad_{e_0}=
\begin{pmatrix}
0 & b & 0 & 0 & 0 \\
0 & 0 & 0 & 0 & 0 \\
0 & 0 & 0 & 0 & 0 \\
0 & 0 & 0 & 0 & -1 \\
0 & 0 & 0 & 1 & 0 
\end{pmatrix},  \quad \quad
U=\ad_{e_4}|_{\text{span}\{e_5,e_1,e_2,e_3\}}=
\begin{pmatrix}
 0 & 0 & 0 & 0 \\
 1 & 0 & 0 & 0 \\
 0 & 0 & 0 & 0 \\
 0 & 0 & 0 & 0 
\end{pmatrix},\]

\[
V=\ad_{e_5}|_{\text{span}\{e_1,e_2,e_3\}}=
\begin{pmatrix}
 0 & 0 & 0 \\
 0 & 0 & 1 \\
 0 & -1 & 0 
\end{pmatrix}.
\]

Then $G=\R e_0\ltimes_\phi \R e_4\ltimes_\psi \R e_5 \ltimes_\rho  H_3$ where $\phi(t)=e^{tD}$, $\psi(t)=e^{tU}$ and $\rho(t)=e^{tV}$. %\footnote{Esta bien eso? No hay que tener algun cuidado con las matrices, que conmuten o algo?}
$H_3$ is $3$-dimensional Heisenberg Lie group, i.e. the Euclidean manifold $\R^3$ equipped with the product
\[(z,x,y)\cdot (z',x',y')=(z+z'+\frac 12 (xy'-x'y),x+x',y+y').\]

For any $k\in\N$ consider the lattice $\Gamma_k$ in $H_3$ given by $\Gamma_k= \frac{1}{2k}\Z\times\Z\times\Z$. Any lattice $\Gamma_k$ is invariant under the subgroup generated by $\rho(\frac{\pi}{2})$, then  
$L=\frac{\pi}{2}\Z\ltimes_\rho\Gamma_k$ is a lattice in $\R e_5 \ltimes_\rho  H_3$.

Now we see that $\psi(\frac{2}{\pi})$ preserves $L$, then we have that $\Lambda=\frac{2}{\pi}\Z\ltimes_\psi L$ is a lattice in $\R e_4\ltimes_\psi \R e_5 \ltimes_\rho  H_3$.

Finally we look for $t_0\neq0$ such that $\Lambda$ is preserved by the subgroup generated by $\phi(t_0)$. It can be shown that if $b=\frac{8}{\pi^3 }$, then $\Lambda$ is invariant under the subgroup generated by  $\phi(\frac{\pi}{2})$. Therefore for this value of $b$,
\[\Gamma=\frac{\pi}{2}\Z \ltimes_\phi\frac{2}{\pi}\Z\ltimes_\psi\frac{\pi}{2}\Z\ltimes_\rho\Gamma_k\]
is a lattice in $G$ for any $k\in\N$. Then $\Gamma\backslash G$ (for any $k\in\Z$) are examples of $6$-dimensional solvmanifolds of type I admitting a LCS structure.

%It can be shown that the nilradical of $\g$ is $\n\simeq\mathbb{R} (\frac{e_0}{a}+e_5) \ltimes (\mathbb{R} e_4 \times \h_3)$, where $\h_3=\operatorname{span}\{e_1,e_2,e_3\}$.
%We can write $\g$ as an almost nilpotent Lie algebra $\g\simeq \mathbb{R} e_0 \ltimes \n$. Next, we show that the corresponding simply connected Lie group admits lattices for some choise of $a,b$. \footnote{expandir y destacar}

\

%\bibliography{mybib}{}

\bibliographystyle{plain}

\end{document}